\documentclass[11pt]{article}
\usepackage[margin=1.2in]{geometry}
    \title{Inspections}
    \date{\today}
   \usepackage{amssymb}
\usepackage{amsthm}
\usepackage{amsmath} 
\usepackage{xcolor}
\usepackage{enumitem}
\usepackage{graphicx}
\usepackage{tikz}
\usepackage{pgfplots}
\pgfplotsset{compat=1.18}
\usepackage{natbib}

\newtheorem{theorem}{Theorem}
\newtheorem{corollary}[theorem]{Corollary}

\newtheorem{proposition}[theorem]{Proposition}
\newtheorem{remark}[theorem]{Remark}

\def\E{\mathbb{E}}
\def\e{\mathrm{e}}

\DeclareMathOperator*{\argmax}{arg\,max}

\begin{document}
\title{Using memory to control admission to unobservable queues}
\author{Refael Hassin  \footnote{School of Mathematical Sciences, Tel-Aviv University, Tel-Aviv, Israel.\\ e-mail: hassin@tauex.tau.ac.il}
   \and  Liron Ravner \footnote{Department of Statistics,  University of Haifa, Haifa, Israel.\\ e-mail: lravner@stat.haifa.ac.il} }

\maketitle

\begin{abstract}
We study admission control to an unobservable M/M/1 queue. A memoryless controller can only randomly thin arrivals (random routing, RR). We show that a gated admission (GA) policy,  blocking arrivals for a fixed period after each admission, stochastically dominates RR at equal throughput, improving social welfare under any sojourn-based cost. We characterize the welfare-maximizing threshold and define the Price of Forgetting as the welfare ratio. This ratio is unbounded even though the absolute welfare gain stays uniformly bounded.
\end{abstract}
 \vspace{1cm}

\section{Introduction}

In the unobservable queue of \citet{EH75}, a Poisson stream of potential customers, with average rate $\Lambda$, arrives at a service system with a single server. The duration of service of a customer is distributed according to an exponential distribution with rate $\mu$. The completion of the service gives the customer a reward $R$, and there is a constant cost $C$ of waiting per unit time.  Arriving customers cannot observe the length of the queue before deciding whether or not to join it, and they base their decision on the known parameters and their belief concerning the way others decide.   In this model, if $\Lambda$ is large, the expected benefit to a customer who joins will be negative if all others join.  The proposed solution is that each new arrival randomly decides to join with such a probability that their expected net benefit is zero. This outcome is clearly undesirable.  

\citet{EH75} propose to impose a fee that reduces customer motivation to join. With this fee, the net benefit of a customer is still zero, but there is a positive benefit to society from each service because the fee itself is considered a transfer payment that does not affect society as a whole. The optimal joining probability that maximizes the total net benefit produced by the system, as well as the corresponding fee and social welfare, is given in Chapter 3 of \cite{HH03}.  If the fee is collected by the server, then it is also the profit maximizing fee because it optimally controls the system and all the value is going to the server.  We observe that by levying this fee, the system is optimized both from the social welfare and profit perspective relative to any other method {\it that controls the Poisson rate of joining the queue}, when the controller and customers cannot observe the queue.

A central planner can adopt a strategy of randomly admitting each
arrival with probability 0.5, alternately admitting every second customer, or admitting a customer only if the time elapsed since the last admission exceeds a given threshold (for exponential arrivals with parameter $\lambda$ the corresponding threshold would be the median $\log 2/\lambda$). The first strategy is more
common in the strategic queueing literature; the latter options may be advantageous but require a stronger means of control, that is, using {\it memory}.

In this paper, we allow the queue controller to use information concerning past arrivals.  Thus, the decision whether to allow an arriving customer to join the queue depends on the history of joining instants. This information provides useful hints about the current length of the queue and enables the decision maker to avoid the accumulation of long queues.   

Suppose that the controller knows the past joining instants.  Clearly, later instants are more significant; the older the joining instant, the smaller its impact on the conditional distribution of the length of the queue.  However, optimally using this information seems both hard to compute and hard to implement.  Therefore, as a first step, we treat a simpler case in which only the most recent joining is considered.  This means a threshold strategy in which, after each instant of joining, further arrivals are blocked for a certain amount of time, denoted $\tau$.  Under this strategy, the system becomes a G/M/1 queue where the arrival process consists of two phases, one of a fixed length and the length of the other is exponentially distributed. 

Our goal is to calculate the expected waiting time in this system as a function of $\tau$, then to compute the threshold value that maximizes the expected social welfare, and finally to compute the price-of-forgetting (PoF), which compares the maximum social welfare with and without memory.  We show that the PoF is significant when the service value is small (in terms of the waiting cost per expected service duration) and the utilization factor is large.  In fact, the PoF is unbounded.

We note that our two-phase G/M/1 system is of independent interest.  It corresponds to a situation where every new arrival blocks entry to the queue for some fixed time.  For example, to prevent unauthorized entry or to maintain privacy.

Our contributions are as follows. First, we show that gated admission stochastically dominates random routing at equal throughput (Proposition~\ref{prop:rho_dominant} and Corollary~\ref{corr:u_dominant}), so that a single bit of memory improves welfare under any sojourn-based cost. Second, we obtain the welfare-maximizing threshold and the optimal social welfare in closed form (Section~\ref{sec:opt_GA}). Third, we prove that the price of forgetting is unbounded and derive its exact divergence rate as the reward approaches the service cost (Proposition~\ref{prop:PoF_divergence}), while showing that the absolute welfare gain stays uniformly bounded (Remark~\ref{rem:bound_diff}). This contrasts with the load-balancing setting of \citet{AG11}, where the PoF is at most~$2$.

\subsection{Literature}

The following research has been carried out on the control of a queueing system using memory.   It may come as a surprise that the most basic M/M/1 system has not been analyzed, which is the subject of this paper.

\citet{Lin03}, \citet{LR03, LR04} and \citet{HS07} consider variations of loss systems with Poisson arrivals, exponentially distributed service duration, and a gatekeeper. Costs are incurred when an arrival is blocked by the gatekeeper and when an admitted customer
must leave the system because the server is busy. The gatekeeper is informed when an admitted customer
finds the server busy, but not when service ends. The discounted cost of the system or its average rate is minimized by a threshold-type policy that blocks entry for a certain period of time after a new arrival.  This property does not always hold for a general distribution of the duration of the service.

\citet{AG11} consider the routing of a Poisson arrival stream to
servers with general heterogeneous service distributions
by a {\it broker} who cannot observe the state of the queues, but knows the previous
dispatching decisions. The ratio of expected delays without and with memory
is {\it the Price of Forgetting} ({\bf PoF}).
When the service is exponentially distributed, {\bf PoF}$\le2$.

\section{Model}\label{sec:model}

Consider a single-server queue with a Poisson arrival process with rate $\lambda$ and exponentially distributed service times with rate $\mu$. For every served customer, there is a reward $R>0$ and a holding cost of $C>0$ per unit of time in the system. We are interested in adopting a static admission policy that maximizes the long-term expected social welfare. By static we mean that the policy is not a function of the queue-length process.

We compare two admission control strategies:
\begin{itemize}
    \item \textbf{Random routing (RR):} Customers are admitted to the system upon arrival with probability $p\in[0,1]$. We denote this policy by $\pi_p$.
    \item \textbf{Gated admission (GA):} After every admission, the arrival stream is blocked for a period of $\tau$ units of time. We denote this policy by $\pi_\tau$. 
\end{itemize}

Under RR the queue process is M/M/1 with arrival rate $\lambda p$. Thus, the stationary expected social welfare (per time unit) is
\begin{align*}
    u(\pi_p)=\lambda p R- \E_p[L]C\ , p\in\left[0,\frac{\mu}{\lambda}\right)\ , 
\end{align*}
where $\E_p[L]$ is the expected queue length at arrival instants given the probability of joining $p$. Clearly, the choice of $p$ must satisfy $\lambda p<\mu$, otherwise the queue is not stable. Let $\rho(p):=\frac{\lambda p}{\mu}$, then  
\begin{align*}
    \E_p[L]=\frac{\rho(p)}{1-\rho(p)}=\frac{\lambda p}{\mu-\lambda p}\ .
\end{align*}

Under GA the queue process is GI/M/1 with iid inter-arrival times distributed as $A\sim\tau+\mathrm{Exp}(\lambda)$. We denote the Laplace-Stieltjes Transform (LST) of $A$ by $a(s):=\E[e^{-sA}]$. The stationary expected social welfare in this case is
\begin{align*}
    u(\pi_\tau)=\lambda_\tau R- C\E_\tau[L]\ , \tau >\max\left\lbrace 0,\frac{1}{\mu}-\frac{1}{\lambda}\right\rbrace , 
\end{align*}
where the effective arrival rate satisfies
\begin{align*}
    \lambda_\tau =\frac{1}{\E[A]}=\frac{\lambda}{1+\lambda\tau}<\mu \ , \forall \tau >\max\left\lbrace 0,\frac{1}{\mu}-\frac{1}{\lambda}\right\rbrace  \ .
\end{align*}
For the G/M/1 queue, the queue length at arrival times is a Geometric random variable with a parameter $\sigma(\tau)$ that is the smallest solution of
\begin{equation}\label{eq:GM1_fixed_point}
    \sigma(\tau)=a(\mu(1-\sigma(\tau)))\ .    
\end{equation}
The corresponding expected queue length (at arrival times) is thus $\frac{\sigma(\tau)}{1-\sigma(\tau)}$. However, as PASTA does not hold, the time average mean, which appears in the mean social welfare, is different. Specifically, the sojourn time in the GI/M/1 system is exponential with parameter $\mu(1-\sigma(\tau))$, hence by Little's law,
\begin{align*}
    \E_\tau[L]=\frac{\lambda_\tau}{\mu(1-\sigma(\tau))}\ .
\end{align*}
As $A$ is the sum of a constant and an exponential random variable, we have
\begin{align*}
    a(s)=\frac{\lambda \e^{-s\tau}}{\lambda+s} \ .
\end{align*}
For the sake of brevity, from now on we use $\rho$ instead of $\rho(p)$ for the RR system and $\sigma$ instead of $\sigma(\tau)$ for the GA system, keeping in mind that these values always depend on the policy parameters.

\section{Dominance of gated admission}\label{sec:dominant_GA}

We next establish that whenever the effective arrival rates for both policies are equal, the expected queue length under GA is strictly smaller under RR. An immediate conclusion of this result is that the optimal GA policy achieves a higher social welfare than the optimal RR policy.

\begin{proposition}\label{prop:rho_dominant}
 If $\lambda_\tau=\lambda p$ and $p\in(0,1)$, then
 \begin{equation}\label{eq:rho_dominant}
      \rho>\sigma\ .
\end{equation}
\end{proposition}
\begin{proof}
Fix $\lambda,\mu>0$ and let $p\in(0,1)$ be such that $\rho(p)=\lambda p/\mu<1$. The GA threshold $\tau> 0$ is chosen so that the arrival rates are equal under both policies;
\begin{equation}\label{eq:equal_rates}
    \lambda p=\lambda_\tau=\frac{\lambda}{1+\lambda\tau}\quad\Longrightarrow\quad
p=\frac{1}{1+\lambda\tau}\ .
\end{equation}
Let $\tilde a(s)$ be the LST of an exponential random variable with rate $\lambda p$. Then, by \eqref{eq:equal_rates},
\begin{align*}
\tilde a(s)=\frac{\lambda p}{\lambda p+s}=\frac{\lambda}{\lambda+s(1+\lambda\tau)} \ .
\end{align*}
For any $s\geq 0$, we have that
\begin{equation}\label{eq:a_inequality}
    a(s)=\frac{\lambda e^{-s\tau}}{\lambda+s}
\leq \frac{\lambda}{(\lambda+s)(1+\tau s)}=\frac{\lambda}{\lambda+s+\lambda s\tau+\tau s^2}
\leq \frac{\lambda}{\lambda+s(1+\lambda\tau)}= \tilde a(s)\ ,
\end{equation}
where the exponential inequality $\e^{-x}\leq \frac{1}{1+x}$ for all $x\geq 0$ was used in the first inequality. Note that the second inequality becomes strict if $s>0$.

As M/M/1 is a special case of G/M/1, $\rho$ is also a solution of a fixed point equation. That is,
\begin{align*}
    \rho=\tilde{a}(\mu(1-\rho))\ ,
\end{align*}
and 
\begin{align*}
    \sigma=a(\mu(1-\sigma))\ .
\end{align*}
Let $g(x):=a(\mu(1-x))$. From the properties of the LST, $g$ is a convex function that increases with $x>0$ and satisfies
\begin{align*}
    g(0)=a(\mu)\in(0,1)\  , g(1)=a(0)=1\ .
\end{align*}
The same is true for the function $f(x):=\tilde{a}(\mu(1-x))$ with the boundary values
\begin{align*}
    f(0)=\tilde{a}(\mu)\geq a(\mu)\  , f(1)=\tilde{a}(0)=1\ ,
\end{align*}
where the inequality follows from \eqref{eq:a_inequality}.
Hence, both functions admit a unique fixed point in $x\in(0,1)$. Suppose that $\sigma=g(\sigma)$, then by \eqref{eq:a_inequality},
\begin{align*}
    \sigma=g(\sigma)<f(\sigma)\ ,
\end{align*}
which implies that the solution of $\rho=f(\rho)$ has to be larger than $\sigma$, i.e., $\rho>\sigma$. 
\end{proof}


Proposition \ref{prop:rho_dominant} implies a stochastic ordering of the queue length (and sojourn time) distributions of both policies when the arrival rate is the same in both systems; the queue is stochastically greater in the RR system. In particular, since $\lambda_\tau=\lambda p$ and $\sigma<\rho$, the time-average expected number in system satisfies $\E_\tau[L]=\lambda_\tau/(\mu(1-\sigma))<\lambda p/(\mu(1-\rho))=\E_p[L]$, so GA incurs a strictly smaller holding cost at equal throughput. This further yields the following corollary. 
\begin{corollary}\label{corr:u_dominant}
    The optimal social welfare achievable by the RR policy is always smaller than the optimal social welfare achievable by the GA policy;
\begin{equation}\label{eq:u_dominant}
    \max_{p\in[0,1]} u(\pi_p) \leq \max_{\tau\geq 0}u(\pi_\tau)\ .
\end{equation}
\end{corollary}
\begin{proof}
    As any RR policy $p\in(0,1)$ can be improved by selecting $\tau$ such that the arrival rate is the same and the expected queue length is shorter, we conclude that \eqref{eq:u_dominant} holds.  Note that it is possible that $p=1$ and $\tau=0$ are optimal, in which case both systems are identical and produce the same social welfare.
\end{proof}

\begin{remark}\label{rem:general_bound}
    The implication of Proposition~\ref{prop:rho_dominant} is, in fact, more general than the dominance of linear social welfare as defined in Section~\ref{sec:model}. The stochastic dominance implies that the GA policy attains a better optimum than the RR policy for any objective function that is monotone decreasing with the sojourn time of customers.
\end{remark}

\section{Optimal gated admission}\label{sec:opt_GA}

We now wish to characterize the optimal GA threshold
\begin{align*}
    \tau^*=\argmax_{\tau\geq 0} u(\pi_\tau)\ . 
\end{align*}
With a slight abuse of the previous notation, we can write the objective function as 

\begin{align*}
    u(\tau)=\lambda_\tau\left(R-C\frac{\E_\tau[L]}{\mu}\right)=\frac{\lambda }{1+\lambda\tau}\left(R- \frac{C}{\mu(1-\sigma(\tau))}\right)\ ,
\end{align*}
where $\sigma(\tau)\in(0,1)$ solves \eqref{eq:GM1_fixed_point}, which in this case can be written as
\begin{align*}
\sigma(\tau)=\frac{\lambda e^{-\mu(1-\sigma(\tau))\tau}}{\lambda+\mu(1-\sigma(\tau))}.
\end{align*}
Equivalently, for any $\sigma\in(0,1)$ one can invert this fixed point equation,
to obtain
\begin{equation}\label{eq:tau_sigma}
\tau(\sigma)=\frac{1}{\mu(1-\sigma)}
\log\left(\frac{\lambda}{\sigma\bigl(\lambda+\mu(1-\sigma)\bigr)}\right).
\end{equation}
Thus, welfare can be expressed as a function of $\sigma$,
\begin{align*}
u(\sigma)=\frac{\lambda\mu(1-\sigma)}{\zeta(\sigma)}\left(R-\frac{C}{\mu(1-\sigma)}\right)=\frac{\lambda\bigl(\mu(1-\sigma)R-C\bigr)}{\zeta(\sigma)}\ ,
\end{align*}
where
\begin{align*}
\zeta(\sigma)=\lambda\log\left(\frac{\lambda}{\sigma(\lambda+\mu(1-\sigma))}\right)+\mu(1-\sigma)\ .
\end{align*}

An interior maximizer $\sigma^*\in(0,\rho)$, with $\rho=\lambda/\mu$, satisfies
\begin{align*}
\mu R\,\zeta(\sigma^*)=\bigl(\mu(1-\sigma^*)R-C\bigr)\,\zeta'(\sigma^*)\ ,
\end{align*}
where
\begin{align*}
\zeta'(\sigma)=-\mu-\frac{\lambda(\lambda+\mu-2\mu\sigma)}{\sigma(\lambda+\mu(1-\sigma))}\ .
\end{align*}
The corresponding optimal threshold is $\tau^*=\tau(\sigma^*)$ as defined in \eqref{eq:tau_sigma}.

It is possible that the optimal policy admits all arriving customers. At $\tau=0$ we have $\sigma=\rho$, $\zeta(\rho)=\mu(\rho-1)$ and $\zeta'(\rho)=\mu(2-\rho)$, so
\begin{align*}
\tau^*=0 \quad\Longleftrightarrow\quad
\frac{ R\mu  }{C}\geq \frac{2-\rho}{(1-\rho)^2}\ .
\end{align*}
Otherwise, if $\lambda R/C<(2-\rho) /(1-\rho)^2$, the maximizer is the unique interior solution $\tau^*=\tau(\sigma^*)$ of the first-order condition above.


\section{Social welfare comparison and the Price of Forgetting}\label{sec:PoF}

From now on we set $\mu=C=1$ without loss of generality and work with the two dimensionless parameters
\begin{align}\label{eq:params}
    \rho := \lambda \in(0,1)\ ,\qquad \nu := R\ ,
\end{align}
where $\rho$ is the server utilization under full admission and $\nu$ is the normalized reward-to-cost ratio.

We define the price of forgetting as
\begin{align}\label{eq:PoF_def}
    \mathrm{PoF}(\rho,\nu) := \frac{u^*_{\mathrm{GA}}(\rho,\nu)}{u^*_{\mathrm{RR}}(\rho,\nu)}\ ,
\end{align}
with the convention $\mathrm{PoF}:=\infty$ when $u^*_{\mathrm{RR}}=0<u^*_{\mathrm{GA}}$.

The social welfare under RR is 
$$u(\pi_p) = \rho p\nu - \frac{\rho p}{1-\rho p}\ ,$$
for $p\in[0,1]$. Equating the derivative to zero yields $(1-\rho p)^2=1/\nu$, therefore, the candidate interior optimum of the RR policy is
\begin{align}\label{eq:p_star}
    p^* = \frac{1}{\rho}\left(1 - \frac{1}{\sqrt{\nu}}\right)\ .
\end{align}
This lies in $(0,1)$ if and only if $1 < \nu < 1/(1-\rho)^2$. Of course, $\nu>1$ is a natural assumption because otherwise the expected cost of a customer's own service time is higher than the expected reward.  The optimal social welfare under RR is thus
\begin{align}\label{eq:u_RR}
u^*_{\mathrm{RR}}(\rho,\nu) = \begin{cases}
    0 & \nu \leq 1\ , \\
    1+\nu-2\sqrt{\nu} & 1 < \nu < 1/(1-\rho)^2\ , \\
    \rho\nu - \rho/(1-\rho) & \nu \geq 1/(1-\rho)^2\ .
\end{cases}
\end{align}

With $\mu=C=1$, the GA welfare becomes $u(\sigma) = \rho((1-\sigma)\nu - 1)/\zeta(\sigma)$, where $\zeta(\sigma)=\rho\log(\rho/(\sigma(\rho+1-\sigma)))+(1-\sigma)>0$.  Since $\zeta>0$, we have $u(\sigma)>0$ if and only if $(1-\sigma)\nu>1$, i.e., $\sigma<1-1/\nu$.  In particular, $u^*_{\mathrm{GA}}=0$ for $\nu\leq 1$: when the reward does not cover a single customer's expected service cost, neither policy benefits from admission.
 
\begin{proposition}\label{prop:cases}
Let $\nu_1:=1/(1-\rho)^2$ and $\nu_2:=(2-\rho)/(1-\rho)^2$. The optimal policies are characterized as follows.
\begin{enumerate}
    \item[\rm(i)] For $\nu \leq 1$, $p^*=0$ and $\tau^*=\infty$; $u^*_{\mathrm{RR}}=u^*_{\mathrm{GA}}=0$.
    \item[\rm(ii)] For $1 < \nu <  \nu_1$, both policies have interior optima; $p^*\in(0,1)$ given by \eqref{eq:p_star} and $\tau^*>0$.
    \item[\rm(iii)] For $\nu_1 \leq \nu < \nu_2$, $p^*=1$ and $\tau^*>0$.
    \item[\rm(iv)] For $\nu \geq  \nu_2$, $p^*=1$ and $\tau^*=0$; $u^*_{\mathrm{RR}}=u^*_{\mathrm{GA}}$ and $\mathrm{PoF}=1$.
\end{enumerate}
In cases \rm{(ii)} and \rm{(iii)}, $\mathrm{PoF}>1$.
\end{proposition}
\begin{proof}
The optimal RR policy follows from \eqref{eq:p_star} and \eqref{eq:u_RR}.  The GA boundary $\nu_2$ is derived in Section~\ref{sec:opt_GA}: $\tau^*=0$ iff $\nu\geq(2-\rho)/(1-\rho)^2$.  For $\nu\leq 1$, $u^*_{\mathrm{GA}}=0$ as shown above.  For $1<\nu<\nu_2$, $u^*_{\mathrm{GA}}>u^*_{\mathrm{RR}}$ by Corollary~\ref{corr:u_dominant}, giving $\mathrm{PoF}>1$.
\end{proof}
 
The PoF is finite in cases (ii) and (iii) and the following proposition shows it is in fact unbounded as $\nu\downarrow 1$. We use the notation $k(\epsilon)\sim h(\epsilon)$ to denote $k(\epsilon)=h(\epsilon)+o(\epsilon)$ as $\epsilon\downarrow 0$.

\begin{proposition}\label{prop:PoF_divergence}
As $\varepsilon:=\nu-1\downarrow 0$, 
\begin{enumerate}
    \item[\rm(a)] $p^*\sim \varepsilon/(2\rho)$ and $u^*_{\mathrm{RR}}=(\sqrt\nu-1)^2\sim \varepsilon^2/4$.
    \item[\rm(b)] $\sigma^*\sim \varepsilon/\bigl(\log(1/\varepsilon)+\log\log(1/\varepsilon)\bigr)$ and $u^*_{\mathrm{GA}}\sim \varepsilon/\bigl(\log(1/\varepsilon)+\log\log(1/\varepsilon)\bigr)$.
    \item[\rm(c)] $\mathrm{PoF}(\rho,\nu)\sim 4/\bigl(\varepsilon(\log(1/\varepsilon)+\log\log(1/\varepsilon))\bigr)\to\infty$.
\end{enumerate}
\end{proposition}
 
\begin{proof}
Part (a) follows directly from \eqref{eq:p_star} and \eqref{eq:u_RR}: $p^*=(1-1/\sqrt{1+\varepsilon})/\rho\sim \varepsilon/(2\rho)$ and $u^*_{\mathrm{RR}}=(\sqrt{1+\varepsilon}-1)^2\sim\varepsilon^2/4$.
 
For part~(b), the numerator of the GA welfare becomes $(1-\sigma)(1+\varepsilon)-1=\varepsilon-\sigma(1+\varepsilon)$, which is positive only for $\sigma<\varepsilon/(1+\varepsilon)\approx\varepsilon$.  For small $\sigma>0$, $\zeta(\sigma)\sim -\rho\log\sigma$ and $\zeta'(\sigma)\sim -\rho/\sigma$.
Substituting $\sigma=\varepsilon t$ into the first-order condition $\nu\zeta+((1-\sigma)\nu-1)\zeta'=0$ and retaining leading-order terms yields
\begin{align*}
\rho\log(1/\varepsilon)+\rho\log(1/t)+O(1)=\frac{\rho(1-t)}{t}\ .
\end{align*}
Setting $L:=\log(1/\varepsilon)$, a first-order balance gives $1/t\sim L$, so $\log(1/t)\sim\log L=\log\log(1/\varepsilon)$. Substituting back yields the refined balance $1/t\sim L+\log L$, hence
\begin{align*}
\sigma^*=\varepsilon t^*\sim \frac{\varepsilon}{L+\log L}=\frac{\varepsilon}{\log(1/\varepsilon)+\log\log(1/\varepsilon)}\ .
\end{align*}
Since $\zeta(\sigma^*)\sim\rho(L+\log L)$ and the numerator $\varepsilon-\sigma^*(1+\varepsilon)\sim\varepsilon$, the optimal welfare is
\begin{align*}
u^*_{\mathrm{GA}}\sim\frac{\rho\varepsilon}{\rho(L+\log L)}=\frac{\varepsilon}{\log(1/\varepsilon)+\log\log(1/\varepsilon)}\ .
\end{align*}
Part (c) follows immediately.
\end{proof}

The divergence of the PoF when $\varepsilon\downarrow 0$ is driven by the different mechanisms of the two policies.  Under RR, the effective arrival rate $\rho p^*\sim\varepsilon/2$ and per-customer cost $1/(1-\rho p^*)\approx 1+\varepsilon/2$ are both determined by $p$, coupling throughput to congestion.  Under GA, the deterministic spacing breaks this coupling: the effective arrival rate $\lambda_{\tau^*}\sim\rho\varepsilon/\!\log(1/\varepsilon)\gg\rho p^*$ is much larger, while the per-customer cost $1/(1-\sigma^*)\approx 1+\varepsilon/\!\log(1/\varepsilon)$ stays just below $\nu=1+\varepsilon$.  The resulting welfare scales as $\varepsilon/\!\log(1/\varepsilon)$ versus $\varepsilon^2/4$, and the ratio diverges.

\begin{remark}\label{rem:bound_diff}
While the PoF is unbounded, the absolute welfare difference $u^*_{\mathrm{GA}}-u^*_{\mathrm{RR}}$ vanishes as $\nu\downarrow 1$ and also as $\nu\to\infty$ (both policies converge to full admission).  Therefore, there is a finite uniform upper bound on the absolute difference over all $(\rho,\nu)$.
\end{remark}

In Tables~\ref{tbl:1} and~\ref{tbl:2} we examine several numerical examples.
\begin{table}[h]
\centering
\begin{tabular}{crrrrrrr}
\hline
Case & $\nu$ & $p^*$ & $\tau^*$ & $\sigma^*$ & $u^*_{\mathrm{RR}}$ & $u^*_{\mathrm{GA}}$ & PoF \\
\hline
(i) &  1 & 0 & $\infty$ & 0 &  0 &  0 & -- \\
(ii) & 2 & 0.586 & 1.00 & 0.162 &  0.172 &  0.269 & 1.569 \\
(iii) & 5 & 1 & 0.12 & 0.444 &  1.500 &  1.513 & 1.009 \\
(iv) & 8 & 1  & 0 & 0.5 & 3 & 3 & 1 \\
\hline
\end{tabular}
\caption{Moderate load $\rho=0.5$, $\nu_1=4$ and $\nu_2=6$.}\label{tbl:1}
\end{table}

\begin{table}[h]
\centering
\begin{tabular}{crrrrrrr}
\hline
Case & $\nu$ & $p^*$ & $\tau^*$ & $\sigma^*$ & $u^*_{\mathrm{RR}}$ & $u^*_{\mathrm{GA}}$ & PoF \\
\hline
(i) &  1 & 0 & $\infty$ & 0 &  0 &  0 & -- \\
(ii) & 2 & 0.366 & 1.26 & 0.175 &  0.172 &  0.315 & 1.834 \\
(ii) & 10 & 0.855 & 0.24 & 0.620 &  4.675 &  4.961 & 1.061 \\
(iii) & 26 & 1 & 0.022 & 0.782 &  16.8 &  16.827 & 1.002 \\
(iv) & 32 & 1  & 0 & 0.8 & 21.6 & 21.6 & 1 \\
\hline
\end{tabular}
\caption{High load $\rho=0.8$, $\nu_1=25$ and $\nu_2=30$.}\label{tbl:2}
\end{table}

In Figure~\ref{fig:PoF} the PoF is plotted for three levels of $\rho$.  The divergence near $\nu=1$ predicted by Proposition~\ref{prop:PoF_divergence} is visible in all three curves.  For $\rho=0.95$ the upper boundary $\nu_2=420$ lies far outside the plotted range, so the curve remains above 1 throughout.

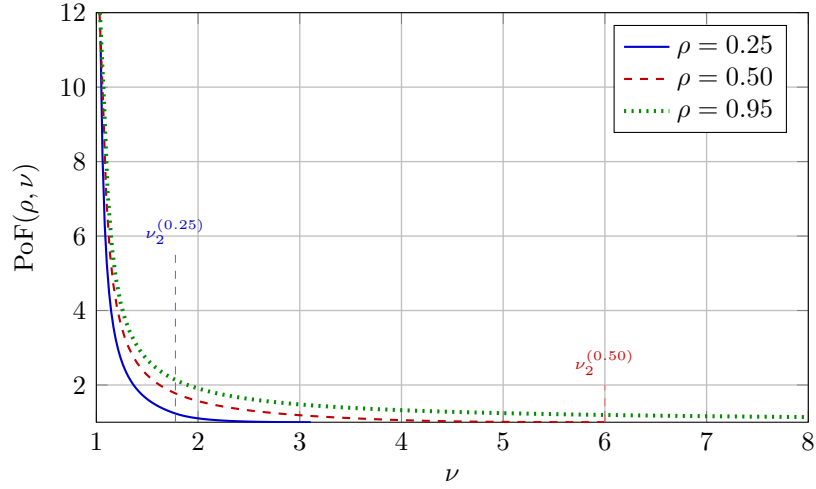
\begin{figure}[ht]
\centering
\begin{tikzpicture}
\begin{axis}[
    width=11cm, height=7cm,
    xmin=1, xmax=8, ymin=1, ymax=12,
    xlabel={$\nu$}, ylabel={$\mathrm{PoF}(\rho,\nu)$},
    legend style={at={(0.97,0.97)}, anchor=north east, font=\small},
    grid=both,
    grid style={line width=0.3pt, draw=gray!30},
    major grid style={line width=0.5pt, draw=gray!50},
    tick label style={font=\small}, label style={font=\small},
    clip=true,
]
\addplot[color=blue!80!black, thick, solid] coordinates {
(1.04140,11.24126)(1.06047,8.15431)(1.08234,6.30609)(1.10216,5.28508)
(1.12490,4.49465)(1.14550,3.98136)(1.16913,3.53893)(1.19054,3.22792)
(1.21509,2.94347)(1.23735,2.73384)(1.26001,2.55526)(1.28600,2.38341)
(1.30955,2.25123)(1.33656,2.12117)(1.36104,2.01918)(1.38911,1.91714)
(1.41455,1.83594)(1.44372,1.75364)(1.47016,1.68738)(1.50049,1.61953)
(1.52797,1.56440)(1.55948,1.50746)(1.58804,1.46084)(1.61713,1.41756)
(1.65048,1.37241)(1.68071,1.33511)(1.71538,1.29599)(1.74679,1.26351)
(1.78282,1.22931)(1.81547,1.20222)(1.85292,1.17580)(1.88685,1.15534)
(1.92577,1.13519)(1.96104,1.11943)(1.99695,1.10543)(2.03814,1.09149)
(2.07547,1.08050)(2.11828,1.06951)(2.15707,1.06083)(2.20156,1.05214)
(2.24188,1.04526)(2.28812,1.03837)(2.33003,1.03294)(2.37809,1.02751)
(2.42164,1.02324)(2.47159,1.01901)(2.51685,1.01570)(2.56294,1.01279)
(2.61581,1.00996)(2.66371,1.00779)(2.71866,1.00572)(2.76845,1.00419)
(2.82555,1.00278)(2.87729,1.00179)(2.93664,1.00096)(2.99042,1.00044)
(3.05210,1.00010)(3.10800,1.00000)
};
\addlegendentry{$\rho=0.25$}
\addplot[color=red!75!black, thick, dashed] coordinates {
(1.03115,12.00000)(1.06496,9.67712)(1.09594,7.12231)(1.13188,5.59016)
(1.16481,4.74102)(1.20300,4.08137)(1.23800,3.65157)(1.27859,3.27999)
(1.31579,3.01804)(1.35893,2.77789)(1.39846,2.60052)(1.43915,2.44915)
(1.48634,2.30325)(1.52958,2.19088)(1.57973,2.08019)(1.62569,1.99331)
(1.67899,1.90634)(1.72784,1.83711)(1.78449,1.76692)(1.83641,1.71043)
(1.89662,1.65259)(1.95180,1.60563)(2.01579,1.55716)(2.07444,1.51752)
(2.13479,1.48075)(2.20479,1.44247)(2.26893,1.41088)(2.34332,1.37783)
(2.41150,1.35044)(2.49057,1.32166)(2.56302,1.29772)(2.64706,1.27247)
(2.72407,1.25140)(2.81339,1.22912)(2.89524,1.21047)(2.97947,1.19283)
(3.07716,1.17411)(3.16668,1.15839)(3.27051,1.14167)(3.36566,1.12761)
(3.47601,1.11263)(3.57714,1.10001)(3.69443,1.08655)(3.80191,1.07520)
(3.92657,1.06308)(4.04080,1.05291)(4.17329,1.04280)(4.29470,1.03504)
(4.41965,1.02831)(4.56456,1.02183)(4.69736,1.01695)(4.85138,1.01235)
(4.99251,1.00898)(5.15621,1.00591)(5.30622,1.00378)(5.48020,1.00200)
(5.63963,1.00091)(5.82455,1.00020)(5.99400,1.00000)
};
\addlegendentry{$\rho=0.50$}
\addplot[color=green!55!black, thick, dotted, line width=1.5pt] coordinates {
(1.03593,12.00000)(1.07549,9.59532)(1.11191,7.13784)(1.15436,5.66012)
(1.19346,4.83908)(1.23902,4.19987)(1.28098,3.78260)(1.32989,3.42124)
(1.37493,3.16613)(1.42742,2.93196)(1.47577,2.75881)(1.52574,2.61091)
(1.58400,2.46823)(1.63764,2.35826)(1.70017,2.24987)(1.75775,2.16475)
(1.82486,2.07950)(1.88666,2.01160)(1.95869,1.94275)(2.02502,1.88732)
(2.10234,1.83056)(2.17354,1.78446)(2.25652,1.73687)(2.33294,1.69795)
(2.41195,1.66184)(2.50404,1.62424)(2.58884,1.59322)(2.68768,1.56075)
(2.77871,1.53384)(2.88480,1.50556)(2.98249,1.48203)(3.09637,1.45722)
(3.20123,1.43650)(3.32345,1.41458)(3.43600,1.39623)(3.55237,1.37887)
(3.68800,1.36043)(3.81290,1.34494)(3.95847,1.32845)(4.09253,1.31458)
(4.24879,1.29978)(4.39268,1.28730)(4.56039,1.27397)(4.71483,1.26271)
(4.89484,1.25066)(5.06062,1.24048)(5.25383,1.22957)(5.43176,1.22033)
(5.61571,1.21150)(5.83012,1.20203)(6.02756,1.19400)(6.25770,1.18538)
(6.46962,1.17806)(6.71663,1.17019)(6.94410,1.16351)(7.20922,1.15632)
(7.45337,1.15021)(7.73794,1.14364)(8.00000,1.13805)
};
\addlegendentry{$\rho=0.95$}
\draw[blue!60, thin, dashed] (axis cs:1.7778,1) -- (axis cs:1.7778,5.5)
  node[above, font=\tiny, blue!80!black] {$\nu_2^{(0.25)}$};
\draw[red!60, thin, dashed] (axis cs:6.0,1) -- (axis cs:6.0,2.0)
  node[above, font=\tiny, red!80!black] {$\nu_2^{(0.50)}$};
\end{axis}
\end{tikzpicture}
\caption{$\mathrm{PoF}(\rho,\nu)$ as a function of $\nu$ for $\rho\in\{0.25,0.50,0.95\}$.  Dashed vertical lines mark $\nu_2=(2-\rho)/(1-\rho)^2$ where applicable.  For $\rho=0.95$, $\nu_2=420$ lies far outside the plotted range.}
\label{fig:PoF}
\end{figure}

Figure~\ref{fig:PoF_zoom} shows a closer view with $\nu\in[1.2,4]$, including $\rho=0.99$.  The PoF curves are ordered by $\rho$: heavier load yields a higher PoF at every $\nu$, reflecting the greater advantage of deterministic spacing when congestion is more severe.  The curves for $\rho=0.95$ and $\rho=0.99$ are nearly indistinguishable, indicating rapid convergence to the $\rho\to 1$ (heavy-traffic) limit.

\begin{figure}[ht]
\centering
\begin{tikzpicture}
\begin{axis}[
    width=11cm, height=7cm,
    xmin=1.2, xmax=4, ymin=1, ymax=5,
    xlabel={$\nu$}, ylabel={$\mathrm{PoF}(\rho,\nu)$},
    legend style={at={(0.97,0.97)}, anchor=north east, font=\small},
    grid=both,
    grid style={line width=0.3pt, draw=gray!30},
    major grid style={line width=0.5pt, draw=gray!50},
    tick label style={font=\small}, label style={font=\small},
    clip=true,
]
\addplot[color=blue!80!black, thick, solid] coordinates {
(1.20000,3.11051)(1.23746,2.73288)(1.28428,2.39383)(1.32174,2.18998)
(1.35920,2.02639)(1.40602,1.86209)(1.44348,1.75429)(1.48094,1.66236)
(1.52776,1.56480)(1.56522,1.49775)(1.60268,1.43857)(1.64950,1.37368)
(1.68696,1.32778)(1.72441,1.28637)(1.77124,1.23996)(1.80870,1.20750)
(1.84615,1.18025)(1.89298,1.15195)(1.93043,1.13298)(1.96789,1.11661)
(2.01472,1.09916)(2.05217,1.08719)(2.08963,1.07668)(2.13645,1.06531)
(2.17391,1.05739)(2.21137,1.05038)(2.25819,1.04271)(2.29565,1.03734)
(2.33311,1.03256)(2.37993,1.02732)(2.41739,1.02363)(2.45485,1.02035)
(2.50167,1.01676)(2.53913,1.01424)(2.57659,1.01201)(2.62341,1.00959)
(2.66087,1.00791)(2.69833,1.00644)(2.74515,1.00487)(2.78261,1.00381)
(2.82007,1.00290)(2.86689,1.00197)(2.90435,1.00138)(2.94181,1.00090)
(2.98863,1.00045)(3.02609,1.00021)(3.06355,1.00007)(3.11037,1.00000)
};
\addlegendentry{$\rho=0.25$}
\addplot[color=red!75!black, thick, dashed] coordinates {
(1.20000,4.12466)(1.23746,3.65730)(1.28428,3.23579)(1.32174,2.98136)
(1.35920,2.77656)(1.40602,2.57028)(1.44348,2.43457)(1.48094,2.31859)
(1.52776,2.19526)(1.56522,2.11034)(1.60268,2.03527)(1.64950,1.95282)
(1.68696,1.89443)(1.72441,1.84167)(1.77124,1.78248)(1.80870,1.73974)
(1.84615,1.70054)(1.89298,1.65588)(1.93043,1.62319)(1.96789,1.59287)
(2.01472,1.55793)(2.05217,1.53208)(2.08963,1.50790)(2.13645,1.47979)
(2.17391,1.45882)(2.21137,1.43908)(2.25819,1.41596)(2.29565,1.39860)
(2.33311,1.38216)(2.37993,1.36281)(2.41739,1.34819)(2.45485,1.33430)
(2.50167,1.31785)(2.53913,1.30538)(2.57659,1.29347)(2.62341,1.27933)
(2.66087,1.26856)(2.69833,1.25824)(2.74515,1.24595)(2.78261,1.23656)
(2.82007,1.22753)(2.86689,1.21675)(2.90435,1.20849)(2.94181,1.20053)
(2.98863,1.19100)(3.02609,1.18368)(3.06355,1.17661)(3.11037,1.16812)
(3.14783,1.16159)(3.18528,1.15528)(3.23211,1.14767)(3.26957,1.14181)
(3.30702,1.13614)(3.35385,1.12929)(3.39130,1.12400)(3.42876,1.11888)
(3.47559,1.11268)(3.51304,1.10789)(3.55050,1.10324)(3.59732,1.09761)
(3.63478,1.09325)(3.67224,1.08901)(3.71906,1.08387)(3.75652,1.07989)
(3.79398,1.07601)(3.84080,1.07131)(3.87826,1.06765)(3.91572,1.06410)
(3.96254,1.05978)(4.00000,1.05642)
};
\addlegendentry{$\rho=0.50$}
\addplot[color=green!55!black, thick, dotted, line width=1.5pt] coordinates {
(1.20000,4.73080)(1.23746,4.21804)(1.28428,3.75464)(1.32174,3.47441)
(1.35920,3.24853)(1.40602,3.02069)(1.44348,2.87060)(1.48094,2.74220)
(1.52776,2.60549)(1.56522,2.51126)(1.60268,2.42789)(1.64950,2.33623)
(1.68696,2.27125)(1.72441,2.21250)(1.77124,2.14651)(1.80870,2.09883)
(1.84615,2.05506)(1.89298,2.00515)(1.93043,1.96858)(1.96789,1.93464)
(2.01472,1.89549)(2.05217,1.86651)(2.08963,1.83938)(2.13645,1.80780)
(2.17391,1.78423)(2.21137,1.76201)(2.25819,1.73598)(2.29565,1.71641)
(2.33311,1.69787)(2.37993,1.67601)(2.41739,1.65950)(2.45485,1.64377)
(2.50167,1.62515)(2.53913,1.61102)(2.57659,1.59751)(2.62341,1.58144)
(2.66087,1.56920)(2.69833,1.55746)(2.74515,1.54346)(2.78261,1.53275)
(2.82007,1.52245)(2.86689,1.51012)(2.90435,1.50067)(2.94181,1.49156)
(2.98863,1.48063)(3.02609,1.47222)(3.06355,1.46411)(3.11037,1.45434)
(3.14783,1.44681)(3.18528,1.43953)(3.23211,1.43075)(3.26957,1.42397)
(3.30702,1.41740)(3.35385,1.40946)(3.39130,1.40332)(3.42876,1.39737)
(3.47559,1.39015)(3.51304,1.38457)(3.55050,1.37914)(3.59732,1.37256)
(3.63478,1.36745)(3.67224,1.36248)(3.71906,1.35645)(3.75652,1.35177)
(3.79398,1.34720)(3.84080,1.34166)(3.87826,1.33734)(3.91572,1.33313)
(3.96254,1.32802)(4.00000,1.32403)
};
\addlegendentry{$\rho=0.95$}
\addplot[color=black!70, thick, dashdotted] coordinates {
(1.20000,4.75765)(1.23746,4.24300)(1.28428,3.77787)(1.32174,3.49658)
(1.35920,3.26982)(1.40602,3.04109)(1.44348,2.89039)(1.48094,2.76148)
(1.52776,2.62421)(1.56522,2.52959)(1.60268,2.44587)(1.64950,2.35382)
(1.68696,2.28856)(1.72441,2.22955)(1.77124,2.16328)(1.80870,2.11538)
(1.84615,2.07142)(1.89298,2.02128)(1.93043,1.98455)(1.96789,1.95045)
(2.01472,1.91112)(2.05217,1.88199)(2.08963,1.85473)(2.13645,1.82300)
(2.17391,1.79931)(2.21137,1.77698)(2.25819,1.75082)(2.29565,1.73115)
(2.33311,1.71251)(2.37993,1.69054)(2.41739,1.67394)(2.45485,1.65813)
(2.50167,1.63941)(2.53913,1.62520)(2.57659,1.61162)(2.62341,1.59546)
(2.66087,1.58315)(2.69833,1.57134)(2.74515,1.55726)(2.78261,1.54648)
(2.82007,1.53613)(2.86689,1.52372)(2.90435,1.51422)(2.94181,1.50505)
(2.98863,1.49405)(3.02609,1.48559)(3.06355,1.47742)(3.11037,1.46759)
(3.14783,1.46001)(3.18528,1.45268)(3.23211,1.44384)(3.26957,1.43702)
(3.30702,1.43040)(3.35385,1.42241)(3.39130,1.41623)(3.42876,1.41023)
(3.47559,1.40297)(3.51304,1.39734)(3.55050,1.39187)(3.59732,1.38524)
(3.63478,1.38010)(3.67224,1.37509)(3.71906,1.36901)(3.75652,1.36429)
(3.79398,1.35969)(3.84080,1.35410)(3.87826,1.34975)(3.91572,1.34551)
(3.96254,1.34035)(4.00000,1.33633)
};
\addlegendentry{$\rho=0.99$}
\end{axis}
\end{tikzpicture}
\caption{$\mathrm{PoF}(\rho,\nu)$ for $\nu\in[1.2,4]$ and $\rho\in\{0.25,0.50,0.95,0.99\}$.  The curves are ordered by $\rho$; the gap between $\rho=0.95$ and $\rho=0.99$ is small, indicating convergence to the heavy-traffic limit.}
\label{fig:PoF_zoom}
\end{figure}
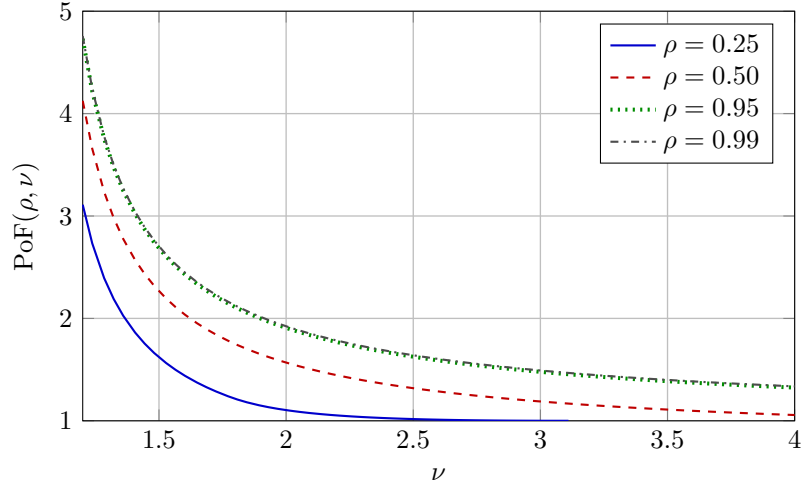

Figure~\ref{fig:welfare} displays the optimal social welfare under both policies for the same three values of $\rho$.

\begin{figure}[ht]
\centering
\pgfplotsset{
  welfare panel/.style={
    width=4.8cm, height=5.2cm, xmin=0, xmax=8, ymin=0,
    xlabel={\small $\nu$},
    grid=both, grid style={line width=0.2pt, draw=gray!25},
    major grid style={line width=0.4pt, draw=gray!40},
    tick label style={font=\footnotesize}, label style={font=\footnotesize},
    title style={font=\small, yshift=-3pt},
    legend style={font=\scriptsize, draw=none, fill=none, at={(0.05,0.97)}, anchor=north west},
    clip=true,
  }
}
\begin{tikzpicture}
\begin{axis}[welfare panel, ymax=1.8, ylabel={\small $u^*$}, title={$\rho=0.25$}, at={(0,0)}]
\addplot[blue!80!black, thick, solid] coordinates {
(0.5000,0.00000)(0.9523,0.00000)(1.1030,0.00252)(1.2538,0.01433)
(1.4045,0.03427)(1.5553,0.06106)(1.7060,0.09373)(1.8945,0.14028)
(2.0452,0.17797)(2.1960,0.21566)(2.3467,0.25335)(2.4975,0.29104)
(2.6482,0.32873)(2.7990,0.36642)(2.9497,0.40410)(3.1005,0.44179)
(3.2513,0.47948)(3.4020,0.51717)(3.5528,0.55486)(3.7035,0.59255)
(3.8543,0.63023)(4.0050,0.66792)(4.1558,0.70561)(4.4950,0.79041)
(4.7965,0.86579)(5.0980,0.94116)(5.3995,1.01654)(5.7010,1.09192)
(6.0025,1.16729)(6.3040,1.24267)(6.6055,1.31805)(6.9447,1.40285)
(7.2462,1.47822)(7.5477,1.55360)(7.8492,1.62898)(8.0000,1.66667)
};
\addlegendentry{RR}
\addplot[red!75!black, thick, dashed] coordinates {
(0.5000,0.00000)(0.9523,0.00000)(1.1030,0.01325)(1.2538,0.03728)
(1.4045,0.06398)(1.5553,0.09249)(1.7060,0.12243)(1.8945,0.16149)
(2.0452,0.19386)(2.1960,0.22713)(2.3467,0.26119)(2.4975,0.29600)
(2.6482,0.33151)(2.7990,0.36766)(2.9497,0.40443)(3.1005,0.44179)
(3.2513,0.47948)(3.4020,0.51717)(3.5528,0.55486)(3.7035,0.59255)
(3.8543,0.63023)(4.0050,0.66792)(4.1558,0.70561)(4.4950,0.79041)
(4.7965,0.86579)(5.0980,0.94116)(5.3995,1.01654)(5.7010,1.09192)
(6.0025,1.16729)(6.3040,1.24267)(6.6055,1.31805)(6.9447,1.40285)
(7.2462,1.47822)(7.5477,1.55360)(7.8492,1.62898)(8.0000,1.66667)
};
\addlegendentry{GA}
\end{axis}
\begin{axis}[welfare panel, ymax=3.2, yticklabels={,,}, title={$\rho=0.50$}, at={(5.2cm,0)}]
\addplot[blue!80!black, thick, solid] coordinates {
(0.5000,0.00000)(0.9523,0.00000)(1.1030,0.00252)(1.2538,0.01433)
(1.4045,0.03427)(1.5553,0.06106)(1.7060,0.09373)(1.8945,0.14168)
(2.0452,0.18500)(2.3467,0.28292)(2.6482,0.39356)(2.9497,0.51478)
(3.2513,0.64501)(3.5528,0.78301)(3.8543,0.92781)(4.0050,1.00251)
(4.3442,1.17211)(4.6457,1.32286)(4.9472,1.47362)(5.2487,1.62437)
(5.5503,1.77513)(5.8518,1.92588)(6.0025,2.00126)(6.3040,2.15201)
(6.6055,2.30276)(6.9447,2.47236)(7.2462,2.62312)(7.5477,2.77387)
(7.8492,2.92462)(8.0000,3.00000)
};
\addplot[red!75!black, thick, dashed] coordinates {
(0.5000,0.00000)(0.9523,0.00000)(1.1030,0.01702)(1.2538,0.05008)
(1.4045,0.08829)(1.5553,0.13017)(1.7060,0.17498)(1.8945,0.23441)
(2.0452,0.28429)(2.3467,0.38942)(2.6482,0.50066)(2.9497,0.61716)
(3.2513,0.73831)(3.5528,0.86363)(3.8543,0.99274)(4.0050,1.05863)
(4.3442,1.20989)(4.6457,1.34765)(4.9472,1.48832)(5.2487,1.63173)
(5.5503,1.77772)(5.8518,1.92616)(6.0025,2.00126)(6.3040,2.15201)
(6.6055,2.30276)(6.9447,2.47236)(7.2462,2.62312)(7.5477,2.77387)
(7.8492,2.92462)(8.0000,3.00000)
};
\end{axis}
\begin{axis}[welfare panel, ymax=4.0, yticklabels={,,}, title={$\rho=0.95$}, at={(10.4cm,0)}]
\addplot[blue!80!black, thick, solid] coordinates {
(0.5000,0.00000)(0.9523,0.00000)(1.1030,0.00252)(1.2538,0.01433)
(1.4045,0.03427)(1.5553,0.06106)(1.7060,0.09373)(1.8945,0.14168)
(2.0452,0.18500)(2.3467,0.28292)(2.6482,0.39356)(2.9497,0.51478)
(3.2513,0.64501)(3.5528,0.78301)(3.8543,0.92781)(4.0050,1.00251)
(4.3442,1.17566)(4.6457,1.33494)(4.9472,1.49876)(5.2487,1.66672)
(5.5503,1.83846)(5.8518,2.01368)(6.0025,2.10251)(6.3040,2.28246)
(6.6055,2.46528)(6.9447,2.67415)(7.2462,2.86247)(7.5477,3.05311)
(7.8492,3.24594)(8.0000,3.34315)
};
\addlegendentry{RR}
\addplot[red!75!black, thick, dashed] coordinates {
(0.5000,0.00000)(0.9523,0.00000)(1.1030,0.01916)(1.2538,0.05788)
(1.4045,0.10375)(1.5553,0.15480)(1.7060,0.21001)(1.8945,0.28387)
(2.0452,0.34627)(2.3467,0.47853)(2.6482,0.61917)(2.9497,0.76686)
(3.2513,0.92059)(3.5528,1.07962)(3.8543,1.24336)(4.0050,1.32683)
(4.3442,1.51823)(4.6457,1.69221)(4.9472,1.86949)(5.2487,2.04980)
(5.5503,2.23293)(5.8518,2.41867)(6.0025,2.51247)(6.3040,2.70181)
(6.6055,2.89337)(6.9447,3.11135)(7.2462,3.30720)(7.5477,3.50489)
(7.8492,3.70432)(8.0000,3.80467)
};
\addlegendentry{GA}
\end{axis}
\end{tikzpicture}
\caption{Optimal social welfare $u^*_{\mathrm{RR}}$ (solid) and $u^*_{\mathrm{GA}}$ (dashed) as functions of $\nu$, for $\rho\in\{0.25,0.50,0.95\}$ (left to right).}
\label{fig:welfare}
\end{figure}
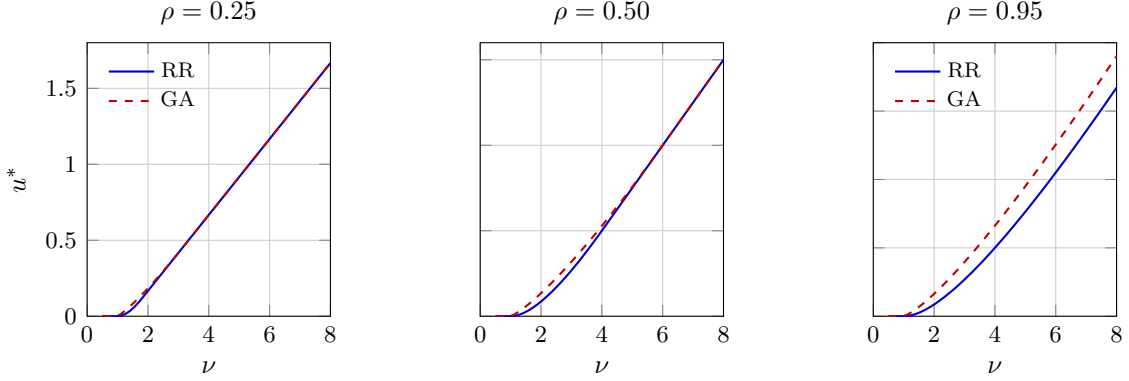

\pagebreak
\bibliographystyle{apalike}
\bibliography{Bibliography}

\end{document}